\documentclass[12pt]{amsart}
\usepackage{amsfonts,amssymb,amscd,amsmath,enumerate,verbatim}
\usepackage[latin1]{inputenc}
\usepackage{amscd}
\usepackage{latexsym}
\usepackage{mathptmx}
\usepackage{multicol}
\input xy
\xyoption{all}
\def\NZQ{\mathbb}               % the font for N,Z,Q,R,C

\def\ZZ{{\NZQ Z}}
\def\RR{{\NZQ R}}

\def\frk{\mathfrak}               % font for "Fraktur"

\def\Phi{{\frk N}}
\def\eb{{\bold e}}

\def\opn#1#2{\def#1{\operatorname{#2}}} % to make operators
\opn\chara{char} \opn\length{\ell} \opn\pd{pd} \opn\rk{rk}
\opn\projdim{proj\,dim} \opn\injdim{inj\,dim} \opn\rank{rank}
\opn\depth{depth} \opn\grade{grade} \opn\height{height}
\opn\embdim{emb\,dim} \opn\codim{codim}

\opn\Tr{Tr} \opn\bigrank{big\,rank}
\opn\superheight{superheight}\opn\lcm{lcm}
\opn\trdeg{tr\,deg}%\emph{
\opn\reg{reg} \opn\lreg{lreg} \opn\ini{in} \opn\lpd{lpd}
\opn\size{size}\opn{\mult}{mult}
\opn\div{div} \opn\Div{Div} \opn\cl{cl} \opn\Cl{Cl}
\opn\Spec{Spec} \opn\Supp{Supp} \opn\supp{supp} \opn\Sing{Sing}
\opn\Ass{Ass} \opn\Min{Min}
\opn\Ann{Ann} \opn\Rad{Rad} \opn\Soc{Soc}
\opn\Syz{Syz} \opn\Im{Im} \opn\Ker{Ker} \opn\Coker{Coker}
\opn\Am{Am} \opn\Hom{Hom} \opn\Tor{Tor} \opn\Ext{Ext}
\opn\End{End} \opn\Aut{Aut} \opn\id{id} \opn\ini{in}

\opn\nat{nat}
\opn\pff{pf}%   \pf exists already
\opn\Pf{Pf} \opn\GL{GL} \opn\SL{SL} \opn\mod{mod} \opn\ord{ord}
\opn\Gin{Gin}
\opn\Hilb{Hilb}\opn\adeg{adeg}\opn\std{std}\opn\ip{infpt}
\opn\Pol{Pol}
\opn\sat{sat}
\opn\Var{Var}
\opn\Gen{Gen}

\opn\aff{aff} \opn\con{conv} \opn\relint{relint} \opn\st{st}
\opn\lk{lk} \opn\cn{cn} \opn\core{core} \opn\vol{vol}
\opn\link{link} \opn\star{star}
\opn\gr{gr}

\def\Cc{{\mathcal C}}
\def\Oc{{\mathcal O}}

\def\Pc{{\mathcal P}}
\def\Qc{{\mathcal Q}}
\def\pot#1#2{#1[\kern-0.28ex[#2]\kern-0.28ex]}

\opn\dirlim{\underrightarrow{\lim}}
\opn\inivlim{\underleftarrow{\lim}}
\let\to=\rightarrow

\def\Implies{\ifmmode\Longrightarrow \else
        \unskip${}\Longrightarrow{}$\ignorespaces\fi}
\def\implies{\ifmmode\Rightarrow \else
        \unskip${}\Rightarrow{}$\ignorespaces\fi}
\def\iff{\ifmmode\Longleftrightarrow \else
        \unskip${}\Longleftrightarrow{}$\ignorespaces\fi}

\let\:=\colon
\newtheorem{Theorem}{Theorem}[section]
\newtheorem{Lemma}[Theorem]{Lemma}
\newtheorem{Corollary}[Theorem]{Corollary}

\newtheorem{Example}[Theorem]{Example}

\let\epsilon\varepsilon
\let\phi=\varphi
\let\kappa=\varkappa
\opn\dis{dis}
\def\pnt{{\raise0.5mm\hbox{\large\bf.}}}

\opn\Lex{Lex}

\begin{document}
\title{The numbers of edges of the order polytope and the chain polytope
of a finite partially ordered set}
\author{% Ginji Hamano, 
Takayuki Hibi, Nan Li, Yoshimi Sahara and Akihiro Shikama}
\thanks{}
% \address{Ginji Hamano,
% Department of Pure and Applied Mathematics,
% Graduate School of Information Science and Technology,
% Osaka University,
% Toyonaka, Osaka 560-0043, Japan}
% \email{g-hamano@cr.math.sci.osaka-u.ac.jp}
\address{Takayuki Hibi,
Department of Pure and Applied Mathematics,
Graduate School of Information Science and Technology,
Osaka University,
Toyonaka, Osaka 560-0043, Japan}
\email{hibi@math.sci.osaka-u.ac.jp}
\address{Nan Li,
Department of Mathematics,
Massachusetts Institute of Technology,
Cambridge, MA 02139, USA}
\email{nan@math.mit.edu}
\address{Yoshimi Sahara,
Department of Pure and Applied Mathematics,
Graduate School of Information Science and Technology,
Osaka University,
Toyonaka, Osaka 560-0043, Japan}
\email{y-sahara@cr.math.sci.osaka-u.ac.jp}
\address{Akihiro Shikama,
Department of Pure and Applied Mathematics,
Graduate School of Information Science and Technology,
Osaka University,
Toyonaka, Osaka 560-0043, Japan}
\email{a-shikama@cr.math.sci.osaka-u.ac.jp}
\thanks{}
\begin{abstract}
Let $P$ be an arbitrary finite partially ordered set.
It will be proved that the number of edges of the order polytope 
$\Oc(P)$ is equal to that of the chain polytope $\Cc(P)$.
Furthermore, it will be shown that
the degree sequence of the finite simple graph 
which is the $1$-skeleton of $\Oc(P)$ is equal to that of $\Cc(P)$ if and only if
$\Oc(P)$ and $\Cc(P)$ are unimodularly equivalent.  
\end{abstract}
\thanks{
\hspace{5.3mm}{\bf Key words and phrases:}
chain polytope, order polytope, partially ordered set.}
\maketitle
\section*{Introduction}
In \cite{Stanley} the combinatorial structure of the order polytope $\Oc(P)$ 
and the chain polytope $\Cc(P)$ of a finite poset (partially ordered set) $P$ 
is studied in detail.
Furthermore, the problem when $\Oc(P)$ and $\Cc(P)$ are unimodularly equivalent 
is solved in \cite{TN}.  
In this paper it is proved that, for an arbitrary
finite poset $P$, the number of edges of the order polytope 
$\Oc(P)$ is equal to that of the chain polytope $\Cc(P)$.
Furthermore, it is shown that
the degree sequence of the finite simple graph 
which is the $1$-skeleton of $\Oc(P)$ is equal to that of $\Cc(P)$ if and only if
$\Oc(P)$ and $\Cc(P)$ are unimodularly equivalent.

\section{Edges of order polytopes and chain polytopes}
Let $P = \{x_{1}, \ldots, x_{d}\}$ be a finite poset.
Given a subset $W \subset P$, we introduce 
$\rho(W) \in \RR^{d}$ by setting $\rho(W) = \sum_{i \in W}\eb_{i}$,
where $\eb_{1}, \eb_{2} \ldots, \eb_{d}$ are the canonical unit coordinate vectors 
of $\RR^{d}$.
In particular $\rho(\emptyset)$ is the origin of $\RR^{d}$.
% A totally ordered subset of $P$ is called a chain of $P$.
A {\em poset ideal} of $P$ is a subset $I$ of $P$ such that,
for all $x_{i}$ and $x_{j}$ with
$x_{i} \in I$ and $x_{j} \leq x_{i}$, one has $x_{j} \in I$.
An {\em antichain} of $P$ is a subset
$A$ of $P$ such that $x_{i}$ and $x_{j}$ belonging to $A$ with $i \neq j$ 
are incomparable.  The empty set $\emptyset$ is a poset ideal
as well as an antichain of $P$.
We say that $x_{j}$ {\em covers} $x_{i}$ if $x_{i} < x_{j}$ and
$x_{i} < x_{k} < x_{j}$ for no $x_{k} \in P$.
A chain $x_{j_{1}} < x_{j_{2}} < \cdots < x_{j_{\ell}}$ of $P$ is called
{\em saturated} if $x_{j_{q}}$ covers $x_{j_{q-1}}$ for $1 < q \leq \ell$.

The {\em order polytope} of $P$ is the convex polytope $\Oc(P) \subset \RR^{d}$
which consists of those $(a_{1}, \ldots, a_{d}) \in \RR^{d}$ such that
$0 \leq a_{i} \leq 1$ for every $1 \leq i \leq d$ together with
\[
a_{i} \geq a_{j}
\]
if $x_{i} \leq x_{j}$ in $P$.

The {\em chain polytope} of $P$ is the convex polytope $\Cc(P) \subset \RR^{d}$
which consists of those $(a_{1}, \ldots, a_{d}) \in \RR^{d}$ such that
$a_{i} \geq 0$ for every $1 \leq i \leq d$ together with
\[
a_{i_{1}} + a_{i_{2}} + \cdots + a_{i_{k}} \leq 1
\]
for every maximal chain $x_{i_{1}} < x_{i_{2}} < \cdots < x_{i_{k}}$ of $P$.

One has $\dim \Oc(P) = \dim \Cc(P) = d$.
The vertices of $\Oc(P)$ is
those $\rho(I)$ for which $I$ is a poset ideal of $P$
(\cite[Corollary 1.3]{Stanley}) and  
the vertices of $\Cc(P)$ is
those $\rho(A)$ for which $A$ is an antichain of $P$
(\cite[Theorem 2.2]{Stanley}).
It then follows that
the number of vertices of $\Oc(P)$ is equal to that of $\Cc(P)$.
Furthermore, the volume of $\Oc(P)$ and that of $\Cc(P)$ are equal to $e(P)/d!$, 
where $e(P)$ is the number of linear extensions of $P$ (\cite[Corollary 4.2]{Stanley}).

In \cite{TNcut}
a characterization of edges of $\Oc(P)$ and those of $\Cc(P)$ is obtained.
Recall that a subposet $Q$ of a finite poset $P$ 
is said to be {\em connected} in $P$ if, for each $x$ and $y$ belonging to $Q$, 
there exists a sequence
$x = x_0, x_1, \ldots, x_s = y$ with each $x_i \in Q$
for which $x_{i-1}$ and $x_{i}$ are comparable in $P$ 
for each $1 \leq i \leq s$. 

\begin{Lemma}
\label{Boston}
Let $P$ be a finite poset.

{\rm (a)} 
Given poset ideals $I$ and $J$ with $I \neq J$, 
the segment combining $\rho(I)$ with $\rho(J)$
is an edge of $\Oc(P)$ if and only if $I \subset J$ and $J \setminus I$ 
is connected in $P$.  

{\rm (b)} 
Given antichains $A$ and $B$ with $A \neq B$, 
the segment combining $\rho(A)$ with $\rho(B)$
is an edge of $\Cc(P)$ if and only if $(A \setminus B) \cup (B \setminus A)$ 
is connected in $P$.  
\end{Lemma}

Let, in general, $G$ be a finite simple graph, i.e., a finite graph
with no loop and with no multiple edge, on the vertex set 
$V(G) = \{ v_1, \ldots, v_n \}$.
The {\em degree} $\deg_G(v_i)$ of each $v_i \in V(G)$ is the number of edges $e$ 
of $G$ with $v_i \in e$.  Let $i_1 \cdots i_n$ denote a permutation of 
$1, \ldots, n$ for which $\deg_G(v_{i_1}) \leq \cdots \leq \deg_G(v_{i_n})$.
The {\em degree sequence} (\cite[p.~216]{Diestel}) of $G$ is the finite sequence 
$(\deg_G(v_{i_1}), \ldots, \deg_G(v_{i_n}))$.

\begin{Example}
\label{EX}
{\em
Let $X$ denote the poset 
\bigskip
$$
\xy 0;/r.2pc/: (0,0)*{\circ}="a";
 (8,10)*{\circ}="b";
 (-8,10)*{\circ}="c";
 (-8,-10)*{\circ}="d";
 (8,-10)*{\circ}="e";
  "a"; "b"**\dir{-};
   "a"; "c"**\dir{-};
    "a"; "d"**\dir{-};
     "a"; "e"**\dir{-};
\endxy
$$

\smallskip

\hspace{6.3cm}
{\rm \bf \small Figure\,\,1}

\bigskip
\medskip

\noindent
Then the degree sequence % (\cite[p.~216]{Diestel})
of the finite simple graph which is the $1$-skeleton of 
$\Oc(X)$ is 
\[
(6,6,6,6,6,6,6,6)
\]
and that of $\Cc(X)$ is
\[
(5,6,6,6,6,6,6,7).
\]
This observation guarantees that, even though the number of edges of
$\Oc(X)$ is equal to that of $\Cc(X)$, one cannot construct a bijection
$\varphi : V(\Oc(X)) \to V(\Cc(X))$, where $V(\Oc(X))$ is the set of vertices of
$\Oc(X)$ and $V(\Cc(X))$ is that of $\Cc(X)$, with the property that,
for $\alpha$ and $\beta$ belonging to $V(\Oc(X))$, 
the segment combining $\alpha$ and $\beta$ is an edge of $\Oc(X)$
if and only if the segment combining $\varphi(\alpha)$ and $\varphi(\beta)$ 
is an edge of $\Cc(X)$.
}
\end{Example}

\section{The number of edges of order polytopes and chain polytopes}
We now come to the main result of the present paper.

\begin{Theorem}
\label{main}
Let $P$ be an arbitrary finite poset.  Then the number of edges of the order polytope
$\Oc(P)$ is equal to that of the chain polytope $\Cc(P)$.
\end{Theorem}

\begin{proof}
Let $\Omega$ denote the set of pairs $(I,J)$, where $I$ and $J$ are poset ideals
of $P$ with $I \neq J$ 
for which $I \subset J$ and $J \setminus I$ 
is connected in $P$.  Let $\Psi$ denote the set of pairs $(A,B)$, 
where $A$ and $B$ are antichains of $P$ with $A \neq B$ for which
$(A \setminus B) \cup (B \setminus A)$ is connected in $P$.

As is stated in the proof of \cite[Lemma 2.3]{TNcut},
if there exist $x, x' \in A$ and $y, y' \in B$ with $x < y$ and $y' < x'$,
then $(A \setminus B) \cup (B \setminus A)$ cannot be connected.  In fact, 
if $(A \setminus B) \cup (B \setminus A)$ is connected, then there exists
a sequence $x = x_0, y_0, x_1, y_1, \ldots, y_s, x_s = x'$ with
each $x_i \in A \setminus B$ and each $b_j \in B \setminus A$ 
such that 
$x_i$ and $y_i$ are comparable for each $i$ 
and that
$y_j$ and $x_{j+1}$ are comparable for each $j$.
Since $x < y$ and since $B$ is an antichain, it follows that $x = x_0 < y_0$.
Then, since $A$ is an antichain, one has $y_0 > x_1$.
Continuing these arguments says that $y_s > x_s = x'$.
However, since $y' < x'$, one has $y' < y_s$, which contradicts the fact 
that $B$ is an antichain.

As a result, each $(A, B) \in \Psi$ can be required to satisfy
either (i) $B \subset A$ or (ii) $b < a$ whenever $a \in A$ and $b \in B$ 
are comparable.
By virtue of Lemma \ref{Boston}, our work is to construct a bijection 
between $\Omega$ and $\Psi$. 

Given $(I,J) \in \Omega$, we associate with
\[
% A = \max(J \setminus I) \cup (\max(I) \cap \max(J)), \, \, \, \, \, 
A = \max(J), \, \, \, \, \,  
B = \min(J \setminus I) \cup (\max(I) \cap \max(J))
\]
with setting $\min(J \setminus I) = \emptyset$ if $|J \setminus I| = 1$,
where, say, $\max(I)$ (resp. $\min(I)$) stands for the set of maximal (resp. minimal)
elements of $I$. 
It then follows that
\begin{eqnarray}
\label{emptyset}
\min(J \setminus I) \cap (\max(I) \cap \max(J)) = \emptyset.
\end{eqnarray}
Now, $A = \max(J)$ is an antichain of $P$.
If $x \in \min(J \setminus I)$ and $y \in \max(I) \cap \max(J)$, then
$x \not\leq y$ since $x \not\in I$ and $y \in I$,
and $y \not\leq x$ since $x \in J$, $x \not= y$ and $y \in \max(J)$.
Hence $B$ is an antichain of $P$.  Furthermore, since
$\max(J) \cap \min(J \setminus I) = \emptyset$,
where
$\min(J \setminus I) = \emptyset$ if $|J \setminus I| = 1$,
it follows that $A \setminus B = \max(J) \setminus \max(I)
= \max(J \setminus I)$ 
and $B \setminus A = \min(J \setminus I)$.
Hence $(A \setminus B) \cup (B \setminus A)$ 
is connected in $P$.  Thus $(A,B) \in \Psi$.

We claim that the above map which associates $(I,J) \in \Omega$ with $(A,B) \in \Psi$
is, in fact, a bijection between $\Omega$ and $\Psi$.

Let $(I,J)$ and $(I',J')$ belong to $\Omega$ with
$\max(J) = \max(J')$ and 
\begin{eqnarray}
\label{Sydney}
\min(J \setminus I) \cup (\max(I) \cap \max(J))
= \min(J' \setminus I') \cup (\max(I') \cap \max(J')).
\end{eqnarray}
Then $J = J'$.  Let $\max(I) \cap \max(J) \neq \max(I') \cap \max(J)$ and,
say, $\max(I) \cap \max(J) \neq \emptyset$.
Let $x \in \max(I) \cap \max(J)$ and $x \not\in \max(I') \cap \max(J)$.
By using (\ref{Sydney}), one has $x \in \min(J \setminus I')$.  
Since $\max(J \setminus I') \cap \min(J \setminus I') = \emptyset$,
where $\min(J \setminus I) = \emptyset$ if $|J \setminus I| = 1$,
there is $y \in \max(J \setminus I')$
with $x < y$.  This is impossible since $x$ and $y$ belong to $\max(J)$.
As a result, 
one has $\max(I) \cap \max(J) = \max(I') \cap \max(J)$.
It then follows from $(\ref{emptyset})$ and $(\ref{Sydney})$ that
$\min(J \setminus I) = \min(J \setminus I')$.
In addition, 
\[
\max(J \setminus I) = \max(J) \setminus \max(I) = 
\max(J) \setminus (\max(I) \cap \max(J)) = \max(J \setminus I').
\]
Since
\[
J \setminus I = \{ \, x \in P \, : 
\, x \leq b, \, \exists \, b \in \max(J \setminus I) \, \} 
\bigcap \{ \, x \in P \, : \, a \leq x, \, \exists \, a \in \min(J \setminus I) \, \},
\]
it follows from
$\min(J \setminus I) = \min(J \setminus I')$
and 
$\max(J \setminus I) = \max(J \setminus I')$
that $J \setminus I = J \setminus I'$.
Hence $I = I'$ and $(I,J) = (I',J')$, as desired.

Let $(A,B)$ belong to $\Psi$.  Let $J$ be the poset ideal of $P$ 
with $\max(J) = A$.
Let $I$ be the poset ideal of $P$ consisting of those $x \in J$ for which
$x \geq y$ for no $y \in B \setminus A$.
In particular, $I = J \setminus \{ x \}$ if $B \subset A$ with
$A \setminus B = \{ x \}$.
Then $\max(J \setminus I) = A \setminus B$ and
$\min(J \setminus I) = B \setminus A$, where
$\min(J \setminus I) = \emptyset$ if $|J \setminus I| = 1$.
Hence $I \subset J$ and $J \setminus I$ is connected in $P$.
Furthermore, 
$B = \min(J \setminus I) \cup (\max(I) \cap \max(J))$, as required.
\, \, \, \, \, \, \, \, \, \, \, \, \, \, \, \, \, \, \, \, \, \, \, \, \,
\, \, \, 
\end{proof}

\section{Degree sequences of $1$-skeletons of order and chain polytopes}
Let $\ZZ^{d \times d}$ denote the set of $d \times d$ integral matrices.
A matrix $A \in \ZZ^{d \times d}$ is {\em unimodular} if $\det(A) = \pm1$. 
Given integral polytopes $\Pc \subset \RR^d$ of dimension $d$ 
and $\Qc \subset \RR^d$ of dimension $d$, 
we say that $\Pc$ and $\Qc$ are {\em unimodularly equivalent} 
if there exists a unimodular matrix $U \in \ZZ^{d \times d}$ 
and an integral vector ${\bf w} \in \ZZ^d$ such that
$Q = f_U(P) + {\bf w}$, where $f_U $ is the linear transformation of $\RR^d$
defined by $U$, i.e., $f_U({\bf v})={\bf v}U$ for all  ${\bf v} \in \RR^d$.

Recall from \cite{TN} that $\Oc(P)$ and $\Cc(P)$ are unimodularly equivalent 
if and only if the poset $X$ of Figure $1$ does not appear as a subposet of $P$.
In consideration of Example \ref{EX}, we now prove the following

\begin{Theorem}
Let $P$ be a finite poset.  Then the degree sequence of the finite simple graph 
which is the $1$-skeleton of $\Oc(P)$ is equal to that of $\Cc(P)$ if and only if
$\Oc(P)$ and $\Cc(P)$ are unimodularly equivalent. 
\end{Theorem}

\begin{proof}
{\bf (``If'')}
\,If $\Oc(P)$ and $\Cc(P)$ are unimodularly equivalent, then 
the $1$-skeleton of $\Oc(P)$ is isomorphic to that of $\Cc(P)$ as finite graphs.
Thus in particular the degree sequence of the $1$-skeleton of $\Oc(P)$ is equal to 
that of $\Cc(P)$, as required. 

\medskip

{\bf (``Only If'')}
\,Let $|P| = d$.
Suppose that $\Oc(P)$ is not unimodularly equivalent to $\Cc(P)$.  
It then follows from \cite[Theorem 2.1]{TN} that
the poset $X$ of Figure $1$ does appear as a subposet of $P$.
Let $X = \{a,b,c,g,h\}$, where $a < c, b < c, c < g$ and $c < h$.
Work with the same notation as in the proof of Theorem \ref{main}.
Write $G_{\Oc(P)}$ for the finite simple graph which is the $1$-skeleton of $\Oc(P)$
and $G_{\Cc(P)}$ for that of $\Cc(P)$.

Let $A \neq \emptyset$ be an antichain of $P$.  Then $(\emptyset, A) \in \Psi$
if and only if $|A| = 1$.  It then follows that the degree of the vertex 
$\rho(\emptyset)$ of $G_{\Cc(P)}$ is equal to $d$. 

We now prove that the degree of each vertex of $G_{\Oc(P)}$ is at least $d+1$.
% For each $x \in P$ we write $I_{x}$ for the poset ideal of $P$ 
% consisting of those $y \in P$ with $y \leq x$
% and $I^{x}$ for the subset of $P$ consisting of those $z \in P$ 
% with $z \geq x$.
Let $I$ be a poset ideal of $P$.  
For each $x \in I$ we write $I'$ for the poset ideal of $P$ consisting of
those $y \in I$ with $y \not\geq x$.  Then $(I', I) \in \Omega$.
For each $x \in P \setminus I$ we write $I'$ for the poset ideal of $P$ 
consisting of those $y \in P$ with either $y \in I$ or $y \leq x$. 
Then $(I, I') \in \Omega$.
As a result, the degree of each vertex of $G_{\Oc(P)}$ is at least $d$.

Since the poset $X = \{a,b,c,g,h\}$ of Figure $1$ does appear as a subposet of $P$,
one has either $c \in I$ or $c \not\in I$.  Let $c \in I$ and $I'$ the poset ideal
of $P$ consisting of those $y \in I$ with neither $y \geq a$ nor $y \geq b$.
Then $(I', I) \in \Omega$.  Let $c \not\in I$ and $I'$ the poset ideal of $P$
consisting of those $y \in P$ with $y \in I$ or $y \leq g$ or $y \leq h$.
Then $(I, I') \in \Omega$.  Hence the degree of each vertex of $G_{\Oc(P)}$ 
is at least $d + 1$, as desired.
%\, \, \, \, \, \, \, \, \, \, \, \, \, \, \, 
%\, \, \, \, \, \, \, \, \, \, \, \, \, \, \, 
%\, \, \, \, \, \, \, \, \, 
\end{proof}

Together with \cite[Corollary 2.3]{TN} it follows that

\begin{Corollary}
\label{Corollary}
Given a finite poset $P$, the following conditions are equivalent:
\begin{enumerate}
\item[(i)] $\Oc(P)$ and $\Cc(P)$ are unimodularly equivalent;
\item[(ii)] $\Oc(P)$ and $\Cc(P)$ are affinely equivalent;
\item[(iii)] $\Oc(P)$ and $\Cc(P)$ have the same $f$-vector {\rm (\cite[p. 12]{HibiRedBook})};
\item[(iv)] The number of facets of $\Oc(P)$ is equal to that of $\Cc(P)$;
\item[(v)] the degree sequence of the finite simple graph 
which is the $1$-skeleton of $\Oc(P)$ is equal to that of $\Cc(P)$;
\item[(vi)] The poset $X$ of Figure $1$ of does not appear
as a subposet of $P$.
\end{enumerate}
\end{Corollary}

% \newpage

{}


\begin{thebibliography}{}
\bibitem{Diestel}
R.~Diestel, ``Graph theory,'' $4$th ed., 
GTM {\bf 173}, Springer, Heidelberg, 2010.
% \bibitem{Hibi}
% T. Hibi,
% Distributive lattices, affine semigroup rings and algebras with straightening laws,
% {\em in} ``Commutative Algebra and Combinatorics''
% (M.\ Nagata and H.\ Matsumura, Eds.),
% Advanced Studies in Pure Math., Volume 11, North--Holland, Amsterdam, 1987, pp. 93 -- 109.
\bibitem{HibiRedBook}
T. Hibi,
``Algebraic combinatorics on convex polytopes,''
Carslaw Publications, Glebe, N.S.W., Australia, 1992.
\bibitem{TN}
T. Hibi and N. Li,
Unimodular equivalence of order and chain polytopes,
Math. Scand., to appear.
\bibitem{TNcut}
T. Hibi and N. Li,
Cutting convex polytopes by hyperplanes,
arXiv:1402.3805.
% \bibitem{OhHiquadratic}
% H. Ohsugi and T. Hibi,
% Toric ideals generated by quadratic binomials,
% {\em J. Algebra} {\bf 218} (1999), 509 -- 527.
% \bibitem{OhHicompressed}
% H. Ohsugi and T. Hibi,
% Convex polytopes all of whose reverse lexicographic initial ideals are squarefree,
% {\em Proc. Amer. Math. Soc.} {\bf 129} (2001), 2541 -- 2546.
% \bibitem{OhHirootsystem}
% H. Ohsugi and T. Hibi,
% Quadratic initial ideals of root systems,
% {\em Proc. Amer. Math. Soc.} {\bf 130} (2002), 1913 -- 1922.
% \bibitem{Schrijver}
% A. Schrijver, ``Theory of Linear and Integer Programming,''
% John Wiley \& Sons, New York, 1986.
\bibitem{Stanley}
R. Stanley,
Two poset polytopes, {\em Discrete Comput. Geom.} {\bf 1} (1986), 9 -- 23.
% \bibitem{StanleyEC}
% R. Stanley,
% ``Enumerative Combinatorics, Volume I,''
% Second Ed., Cambridge University Press, Cambridge, 2012.
% \bibitem{Sturmfels}
% B. Sturmfels, ``Gr\"obner Bases and Convex Polytopes,''
% Amer. Math. Soc., Providence, RI, 1995.
\end{thebibliography}
\end{document}